\documentclass[a4paper,american]{amsart}
\usepackage{mathptmx}
\usepackage[T1]{fontenc}
\usepackage[latin9]{inputenc}
\usepackage[english]{babel}
\usepackage{refstyle}
\usepackage{units}
\usepackage{amsmath}
\usepackage{amsthm}
\usepackage{amssymb}
\usepackage[unicode=true,pdfusetitle,
 bookmarks=true,bookmarksnumbered=false,bookmarksopen=false,
 breaklinks=false,pdfborder={0 0 1},backref=false,colorlinks=false]
 {hyperref}
\hypersetup{
 bookmarksopen}

\makeatletter

\usepackage{verbatim}
\def\AS{\\ }
\def\SU{\substack}

\AtBeginDocument{\providecommand\secref[1]{\ref{sec:#1}}}
\AtBeginDocument{\providecommand\thmref[1]{\ref{thm:#1}}}
\AtBeginDocument{\providecommand\lemref[1]{\ref{lem:#1}}}
\AtBeginDocument{\providecommand\propref[1]{\ref{prop:#1}}}
\pdfpageheight\paperheight
\pdfpagewidth\paperwidth
\RS@ifundefined{subref}
  {\def\RSsubtxt{section~}\newref{sub}{name = \RSsubtxt}}
  {}
\RS@ifundefined{thmref}
  {\def\RSthmtxt{theorem~}\newref{thm}{name = \RSthmtxt}}
  {}
\RS@ifundefined{lemref}
  {\def\RSlemtxt{lemma~}\newref{lem}{name = \RSlemtxt}}
  {}

\theoremstyle{plain}
\numberwithin{equation}{section}
\numberwithin{figure}{section}
\numberwithin{table}{section}
    \newtheorem{stthm}{\protect\theoremname}[section]
    
    \newtheorem{spprop}{\protect\propositionname}[section]

 \newref{sec}{%
 name = \RSsectxt,
 names = \RSsecstxt,
 lsttxt = \RSlsttxt,
 lsttwotxt = \RSlsttwotxt,
 refcmd = \ref{#1}}
 \newref{thm}{%
 name = \theoremname~,
 lsttxt = \RSlsttxt,
 lsttwotxt = \RSlsttwotxt,
 refcmd = {\ref{#1}}}
 \usepackage{enumitem}
 \usepackage{enumitem}
 \newref{prop}{%
 name = \propositionname~,
 lsttxt = \RSlsttxt,
 lsttwotxt = \RSlsttwotxt,
 refcmd = {\ref{#1}}}
 \theoremstyle{plain}
 \newtheorem*{cor*}{\protect\corollaryname}
 \newref{cor}{%
 name = \corollaryname~,
 lsttxt = \RSlsttxt,
 lsttwotxt = \RSlsttwotxt,
 refcmd = {\ref{#1}}}
 \theoremstyle{plain}
 \newtheorem{sllem}{\protect\lemmaname}[section]
 \newref{lem}{%
 name = \lemmaname~,
 lsttxt = \RSlsttxt,
 lsttwotxt = \RSlsttwotxt,
 refcmd = {\ref{#1}}}
 \newref{sec}{%
 name = \RSsectxt,
 names = \RSsecstxt,
 lsttxt = \RSlsttxt,
 lsttwotxt = \RSlsttwotxt,
 refcmd = \ref{#1}}
 \theoremstyle{remark}
 \newtheorem*{rem*}{\protect\remarkname}
 \newref{rem}{%
 name = \remarkname~,
 lsttxt = \RSlsttxt,
 lsttwotxt = \RSlsttwotxt,
 refcmd = {\ref{#1}}}

\newcommand\mynobreakpar{\par\nobreak\@afterheading}

\makeatother

 \addto\captionsamerican{\renewcommand{\corollaryname}{Corollary}}
 \addto\captionsamerican{\renewcommand{\lemmaname}{Lemma}}
 \addto\captionsamerican{\renewcommand{\propositionname}{Proposition}}
 \addto\captionsamerican{\renewcommand{\remarkname}{Remark}}
 \addto\captionsamerican{\renewcommand{\theoremname}{Theorem}}
 \addto\captionsenglish{\renewcommand{\corollaryname}{Corollary}}
 \addto\captionsenglish{\renewcommand{\lemmaname}{Lemma}}
 \addto\captionsenglish{\renewcommand{\propositionname}{Proposition}}
 \addto\captionsenglish{\renewcommand{\remarkname}{Remark}}
 \addto\captionsenglish{\renewcommand{\theoremname}{Theorem}}
 \providecommand{\corollaryname}{Corollary}
 \providecommand{\lemmaname}{Lemma}
 \providecommand{\propositionname}{Proposition}
 \providecommand{\remarkname}{Remark}
 \providecommand{\theoremname}{Theorem}

\begin{document}
\selectlanguage{english}%
\def\rightmark{WEIGHTED ESTIMATES FOR SOLUTIONS OF THE $\bar\partial$ EQUATION AND APPLICATIONS}
\def\leftmark{P. CHARPENTIER, Y. DUPAIN \& M. MOUNKAILA}
\def\RSsectxt{Section~}%

\selectlanguage{american}%

\title{Weighted Estimates for Solutions of the $\overline{\partial}$-equation
for Lineally Convex Domains of Finite Type and Applications to weighted
Bergman projections}

\author{P. Charpentier, Y. Dupain \& M. Mounkaila}
\begin{abstract}
In this paper we obtain sharp weighted estimates for solutions of
the $\overline{\partial}$-equation in a lineally convex domains of
finite type. Precisely we obtain estimates in spaces of the form $L^{p}(\Omega,\delta^{\gamma})$,
$\delta$ being the distance to the boundary, with gain on the index
$p$ and the exponent $\gamma$. These estimates allow us to extend
the $L^{p}(\Omega,\delta^{\gamma})$ and lipschitz regularity results for weighted
Bergman projection obtained in \cite{CDM} for convex domains to more
general weights.
\end{abstract}



\address{P. Charpentier, Universit\'e Bordeaux I, Institut de Math\'ematiques
de Bordeaux, 351, Cours de la Lib\'eration, 33405, Talence, France}

\address{M. Mounkaila, Universit\'e Abdou Moumouni, Facult\'e des Sciences,
B.P. 10662, Niamey, Niger}

\email{P. Charpentier: philippe.charpentier@math.u-bordeaux1.fr}

\email{M. Mounkaila: modi.mounkaila@yahoo.fr}

\maketitle

\section{Introduction}

The study of the regularity of the Bergman projection onto holomorphic
functions in a given Hilbert space is a very classical subject. When
the Hilbert space is the standard Lebesgue $L^{2}$ space on a smoothly
bounded pseudoconvex domain $\Omega$ in $\mathbb{C}^{n}$, many results
are known and there is a very large bibliography.

When the Hilbert space is a weighted $L^{2}$ space on a smoothly
bounded pseudoconvex domain $\Omega$ in $\mathbb{C}^{n}$, it is
well known for a long time that the regularity of the Bergman projection
depends strongly on the weight (\cite{Kohn-defining-function}, \cite{Bar92},
\cite{Christ96}). Until last years few results where known (see \cite{FR75},
\cite{Lig89}, \cite{BG95}, \cite{CDC97}) but recently some positive
and negative results where obtained by several authors (see for example \cite{Zey11},
\cite{Zey12}, \cite{Zey13a}, \cite{Zey13b}, \cite{CDM}, \cite{CPDY},
\cite{CZ}, \cite{Zey} and references therein).

In this paper we are interested in some generalization of the result
obtained in \cite{CDM} for convex domains of finite type.

Let $\Omega$ be a convex domain of finite type in $\mathbb{C}^{n}$.
Let $g$ be a gauge function for $\Omega$ and define $\rho_{0}=g^{4}e^{1-\nicefrac{1}{g}}-1$.
Let $P_{\omega_{0}}$ be the Bergman projection of the space $L^{2}\left(\Omega,\omega_{0}\right)$,
where $\omega_{0}=\left(-\rho_{0}\right)^{r}$, $r\in\mathbb{Q}_{+}$.
Then in \cite[Theorem 2.1]{CDM} we proved that $P_{\omega_{0}}$
maps continuously the spaces $L^{p}\left(\Omega,\delta_{\Omega}^{\beta}\right)$,
$p\in\left]1,+\infty\right[$, $0<\beta+1\leq p(r+1)$, into themselves,
$\delta_{\Omega}$ being the distance to the boundary of $\Omega$.
Here we consider a weight $\omega$ which is a non negative rational
power of a $\mathcal{C}^{2}$ function in $\overline{\Omega}$ equivalent
to the distance to the boundary and we prove that the Bergman projection
$P_{\omega}$ of the Hilbert space $L^{2}\left(\Omega,\omega\right)$
maps continuously the spaces $L^{p}\left(\Omega,\delta_{\partial\Omega}^{\beta}\right)$,
$p\in\left]1,+\infty\right[$, $0<\beta+1\leq r+1$ into themselves
and the lipschitz spaces $\Lambda_{\alpha}(\Omega)$, $0<\alpha\leq\nicefrac{1}{m}$,
into themselves.

This result is obtained comparing the operators $P_{\omega_{0}}$
and $P_{\omega}$ with the method described in \cite{CPDY}. To do
it we need to have weighted $L^{p}\left(\Omega,\delta_{\Omega}^{\gamma}\right)$
estimates with appropriate gains on the index $p$ and on the power
$\gamma$ for solution of the $\overline{\partial}$-equation.

This is done, with sharp estimates, for a general lineally convex
domain of finite type using the method introduced in \cite{CDMb},
which overcomes the fact that the Diederich-Fornaess support function
is only locally defined and that it is not possible do make a division
with good estimates in non convex domains.

Our results extend the results (without weights) obtained for convex
domains of finite type by A. Cumenge in \cite{Cumenge-estimates-holder}
and \cite{Cumenge-Navanlinna-convex} and B. Fisher in \cite{MR1815835}
(see also T. Hefer \cite{Hef02}).

\section{Notations and main results}

Throughout this paper we will use the following general notations:
\begin{itemize}
\item $\Omega$ is a smoothly bounded lineally convex domain of finite type
$m$ in $\mathbb{C}^{n}$ (see \cite{CDMb} for a precise definition).
\item $\rho$ is a smooth defining function of $\Omega$ such that, for
$\delta_{0}$ sufficiently small, the domains $\Omega_{t}=\left\{ \rho(z)<t\right\} $,
$-\delta_{0}\leq t\leq\delta_{0}$, are all lineally convex of finite
type $m$.
\item $\delta_{\Omega}$ denotes the distance to the boundary of $\Omega$.
\item For any real number $\gamma>-1$, we denote by $L^{p}\left(\Omega,\delta_{\Omega}^{\gamma}\right)$
the $L^{p}$-space on $\Omega$ for the measure $\delta_{\Omega}^{\gamma}(z)d\lambda(z)$, $\lambda$ being the Lebesgue measure.
\end{itemize}

Our first results give sharp $L^{q}\left(\Omega,\delta_{\Omega}^{\gamma'}\right)$
estimates for solutions of the $\overline{\partial}$-equation in
$\Omega$ with data in $L^{p}\left(\Omega,\delta_{\Omega}^{\gamma}\right)$:
\begin{stthm}
\label{thm:d-bar-q-gamma'-p-gamma-lip}Let $N$ be a positive large
integer. let $\gamma$ and $\gamma'$ be two real numbers such that
$\max\left\{ -1,\gamma-\nicefrac{1}{m}\right\} \leq\gamma'\leq\gamma\leq N-2$.
Then there exists a linear operator $T$, depending on $\rho$ and
$N$, such that, for any $\overline{\partial}$-closed $\left(0,r\right)$-form
with coefficients in $L^{p}\left(\Omega,\delta_{\Omega}^{\gamma}\right)$,
$p\in\left[1,+\infty\right]$, $Tf$ is a solution of the equation
$\overline{\partial}(Tf)=f$ satisfying the following estimate:
\begin{enumerate}
\item If $1\leq p<\frac{m(\gamma'+n)+2}{1-m(\gamma-\gamma')}$, $T$ maps
continuously the space of $\overline{\partial}$-closed forms with
coefficients in $L^{p}\left(\Omega,\delta_{\Omega}^{\gamma}\right)$
into the space of forms whose coefficients are in $L^{q}\left(\Omega,\delta_{\Omega}^{\gamma'}\right)$
with $\frac{1}{q}=\frac{1}{p}-\frac{1-m(\gamma-\gamma')}{m(\gamma'+n)+2}$;
\item If $p=m(\gamma+n)+2$, $T$ maps continuously the space of $\overline{\partial}$-closed
forms with coefficients in $L^{p}\left(\Omega,\delta_{\Omega}^{\gamma}\right)$
into the space of forms whose coefficients are in $BMO(\Omega)$;
\item If $p\in\left]m(\gamma+n)+2,+\infty\right]$, then $T$ maps continuously
the space of $\overline{\partial}$-closed forms with coefficients
in $L^{p}\left(\Omega,\delta_{\Omega}^{\gamma}\right)$ into the space
of forms whose coefficients are in the lipschitz space $\Lambda_{\alpha}(\Omega)$
with $\alpha=\frac{1}{m}\left[1-\frac{m(\gamma+n)+2}{p}\right]$.
\end{enumerate}
\end{stthm}

Note that, if $\gamma'<\gamma$, then $\frac{m(\gamma'+n)+2}{1-m(\gamma-\gamma')}>m(\gamma+n)+2$,
and (3) is sharper than (1). Moreover, without weights, these estimates
are known to be sharp (see \cite{CKM93}).

\medskip{}

The two next propositions, which are immediate corollaries of the
theorem, will be used in the last section:
\begin{spprop}
\label{prop:d-bar-gain-weight}There exists a constant $\varepsilon_{0}>0$
such that, for all large integer $N$ and all $-1<\gamma\leq N-2$,
there exists a linear operator $T$ solving the $\overline{\partial}$-equation
in $\Omega$ such that, for all $p\in\left[1,+\infty\right[$, there
exists a constant $C_{N,p}>0$ such that for all $\overline{\partial}$-closed
$\left(0,r\right)$-form $f$, $1\leq r\leq n-1$, on $\overline{\Omega}$,
we have
\[
\int_{\Omega}\left|Tf\right|^{p}\delta_{\Omega}^{\gamma}d\lambda\leq C_{N,p}\int_{\Omega}\left|f\right|^{p}\delta_{\Omega}^{\gamma+\varepsilon_{0}}d\lambda.
\]

\end{spprop}

\begin{spprop}
\label{prop:d-bar-gain-exponent}There exist a linear operator $T$
solving the $\overline{\partial}$-equation in $\Omega$ and a constant
$\varepsilon_{0}>0$ such that, for all $-1<\gamma\leq N-2$ and all
$p\in\left[1,+\infty\right[$, there exists a constant $C_{N,p}>0$
such that for all $\overline{\partial}$-closed $\left(0,r\right)$-form
$f$, $1\leq r\leq n-1$, we have
\[
\int_{\Omega}\left|Tf\right|^{p+\varepsilon_{0}}\delta_{\Omega}^{\gamma}d\lambda\leq C_{N,p}\int_{\Omega}\left|f\right|^{p}\delta_{\Omega}^{\gamma}d\lambda.
\]

\end{spprop}

\medskip{}

Our last estimate for solutions of the $\overline{\partial}$-equation
is a generalization to lineally convex domains of an estimate obtained
by A. Cumenge (\cite{Cumenge-Navanlinna-convex}) for convex domains
of finite type:
\begin{stthm}
\label{thm:d-bar-for-Nev}For all $\alpha>0$ there exists a constant
$C>0$ such that, for all smooth $\overline{\partial}$-closed $\left(0,r\right)$-form
$f$, $1\leq r\leq n-1$, on $\overline{\Omega}$, there exists a
solution of the equation $\overline{\partial}u=f$, continuous on
$\overline{\Omega}$ such that
\[
\int_{\Omega}\left|u\right|\delta_{\Omega}^{\alpha-1}d\lambda\leq C\frac{1}{\alpha}\int_{\Omega}\left\Vert f\right\Vert _{k}\delta_{\Omega}^{\alpha}d\lambda,
\]
where the norm $\left\Vert f\right\Vert _{k}$ was introduced in \cite{Bruna-Charp-Dupain-Annals}
(see \cite{CDMb} for details, the definition is recalled in \secref{Proof-of-thm-est-d-bar-Nev}).
\end{stthm}

Note that the estimate given by \thmref{d-bar-q-gamma'-p-gamma-lip}
when $p=q=1$ (and then $\gamma'=\gamma-\nicefrac{1}{m}$) is weaker
than the one given above.

An immediate application of this last estimate is the characterization
of the zero sets of the weighted Nevanlinna classes (called Nevanlinna-Djrbachian
classes in \cite{Cumenge-Navanlinna-convex}) obtained by A. Cumenge
for convex domains:
\begin{stthm}
A divisor $\mathcal{D}$ in $\Omega$ can be defined by a holomorphic
function satisfying $\int_{\Omega}\ln^{+}\left|f\right|\delta_{\Omega}^{\alpha-1}d\lambda<+\infty$,
$\alpha>0$, if and only if it satisfy the generalized Blaschke condition
$\int_{\mathcal{D}}\delta^{\alpha+1}d\lambda_{2n-2}<+\infty$.
\end{stthm}

As the proof of such result using \thmref{d-bar-for-Nev} is very
classical we will not give any detail on it in this paper.

\medskip{}

The two propositions \ref{prop:d-bar-gain-weight} and \ref{prop:d-bar-gain-exponent}
will be used to generalize some estimates obtained for weighted Bergman
projections of convex domains of finite type in \cite{CDM}:
\begin{stthm}
\label{thm:estimates-bergman}Let $D$ be smoothly bounded convex
domain of finite type in $\mathbb{C}^{n}$. Let $\chi$ be any $\mathcal{C}^{2}$
non negative function in $\overline{D}$ which is equivalent to the
distance $\delta_{D}$ to the boundary of $D$ and let $\eta$ be
a strictly positive $\mathcal{C}^{1}$ function on $\overline{D}$.
Let $P_{\omega}$ be the (weighted) Bergman projection of the Hilbert
space $L^{2}\left(D,\omega\right)$ where $\omega=\eta\chi^{r}$ with
$r$ a non negative rational number.Then:
\begin{enumerate}
\item For $p\in\left]1,+\infty\right[$ and $-1<\beta\leq r$, $P_{\omega}$
maps continuously $L^{p}\left(D,\delta_{D}^{\beta}\right)$ into itself.
\item For $0<\alpha\leq\nicefrac{1}{m}$ $P_{\omega}$ maps continuously
the Lipschitz space $\Lambda_{\alpha}(D)$ into itself.
\end{enumerate}
\end{stthm}

This theorem combined with \thmref{d-bar-q-gamma'-p-gamma-lip} extends
to weighted situations the Corollary 1.3 of \cite{Cumenge-estimates-holder}
\begin{cor*}
Under the assumptions of \thmref{estimates-bergman}, the solution
of the equation $\overline{\partial}u=f$ which is orthogonal to holomorphic
functions in $L^{2}(D,\omega)$ where $f$ is a $\left(0,1\right)$-form
$\overline{\partial}$-closed with coefficients in $L^{p}(\Omega,\delta_{\Omega}^{\gamma})$,
$-1<\gamma$, belongs to:
\begin{enumerate}
\item $L^{q}(D,\delta_{D}^{\gamma'})$, with $\frac{1}{q}=\frac{1}{p}-\frac{1-m(\gamma-\gamma')}{m(\gamma'+n)+2}$
and $\max\left\{ -1,\gamma-\nicefrac{1}{m}\right\} <\gamma'\leq\gamma$,
if $\gamma'\leq r$, $1\leq p<\frac{m(\gamma'+n)+2}{1-m(\gamma-\gamma')}$,
and $q>1$;
\item $\Lambda_{\alpha}(D)$, with $\alpha=\frac{1}{m}\left[1-\frac{m(\gamma+n)+2}{p}\right]$,
if $p\in\left]m(\gamma+n)+2,+\infty\right]$.
\end{enumerate}
\end{cor*}

\section{Proofs of theorems \ref{thm:d-bar-q-gamma'-p-gamma-lip} and \ref{thm:d-bar-for-Nev}}

First of all by standard regularization procedure, it suffices to
prove theorems \ref{thm:d-bar-q-gamma'-p-gamma-lip}, and \ref{thm:d-bar-for-Nev}
for forms smooth in $\overline{\Omega}$.

To solve the $\overline{\partial}$-equation on a lineally convex
domain of finite type, we use the method introduced in \cite{CDMb}.
We now briefly recall the notations and main results from that work.

If $f$ is a smooth $\left(0,r\right)$-form $\overline{\partial}$-closed,
the following formula was established
\[
f(z)=\left(-1\right)^{q+1}\overline{\partial_{z}}\left(\int_{\Omega}f(\zeta)\wedge K_{N}^{1}(z,\zeta)\right)-\int_{\Omega}f(\zeta)\wedge P_{N}(z,\zeta),
\]
where $K_{N}^{1}$ (resp. $P_{N}$) is the component of a kernel $K_{N}$
(formula (2.7) of \cite{CDMb}) of bi-degree $\left(0,r\right)$ in
$z$ and $\left(n,n-r-1\right)$ in $\zeta$ (resp. $\left(0,r\right)$
in $z$ and $\left(n,n-r\right)$ in $\zeta$) constructed with the
method of \cite{BA82} using the Diederich-Fornaess support function
constructed in \cite{Diederich-Fornaess-Support-Func-lineally-cvx}
(see also Theorem 2.2 of \cite{CDMb}) and the function $G(\xi)=\frac{1}{\xi^{N}}$
with a sufficiently large number $N$ (instead of $G(\xi)=\frac{1}{\xi}$
in formula (2.7) of \cite{CDMb}).

Then, the form $\int_{\Omega}f(\zeta)\wedge P_{N}(z,\zeta)$ is $\overline{\partial}$-closed
and the operator $T$ solving the $\overline{\partial}$-equation
in theorems \ref{thm:d-bar-q-gamma'-p-gamma-lip} and \ref{thm:d-bar-for-Nev}
is defined on smooth forms by
\[
Tf(z)=\int_{\Omega}f(\zeta)\wedge K_{N}^{1}(z,\zeta)-\overline{\partial}^{*}\mathcal{N}\left(\int_{\Omega}f(\zeta)\wedge P_{N}(z,\zeta)\right),
\]
where $\overline{\partial}^{*}\mathcal{N}$ is the canonical solution
of the $\overline{\partial}$-equation derived from the theory of
the $\overline{\partial}$-Neumann problem on pseudoconvex domains
of finite type.

This formula is justified by the fact that, when the coefficients
of $f$ are in $L^{1}\left(\Omega,\delta_{\Omega}^{\gamma}\right)$
($\gamma>-1$) then, given a large integer $s$, if $N$ is chosen
sufficiently large, the coefficients of the form $\int_{\Omega}f(\zeta)\wedge P_{N}(z,\zeta)$
are in the Sobolev space $L_{s}^{2}(\Omega)$. More precisely, it
is clear that lemmas 2.2 and 2.3 of \cite{CDMb} remains true with
weighted estimates depending on the choice of $N$:
\begin{sllem}
For $r\geq1$ and $\gamma\leq N$, all the $z$-derivatives of $P_{N}(z,\zeta)\left(-\rho(\zeta)\right)^{-\gamma}$
are uniformly bounded in $\overline{\Omega}\times\overline{\Omega}$,
and, for each positive integer $s$, there exists a constant $C_{s,N,\gamma}$
such that, if $f$ is $\left(0,r\right)$-form with coefficients in
$L^{1}(\Omega,\delta_{\Omega}^{\gamma})$,
\[
\left\Vert \int_{\Omega}f(\zeta)\wedge P_{N}(z,\zeta)\right\Vert _{L_{s}^{2}(\Omega)}\leq C_{s,N,\gamma}\left\Vert f\right\Vert _{L^{1}(\Omega,\delta_{\Omega}^{\gamma})}.
\]

\end{sllem}

As $\Omega$ is assumed to be smooth and of finite type, the regularity
results of the $\overline{\partial}$-Neumann problem (\cite{Kohn-Nirenberg-1965}
and \cite{Catlin-Est.-Sous-ellipt.})
\begin{sllem}
For $r\geq1$ and $-1<\gamma\leq N$, for each positive integer $s$,
if $f$ is a $\overline{\partial}$-closed $\left(0,r\right)$-form
with coefficients in $L^{1}(\Omega,\delta_{\Omega}^{\gamma})$ and
$g=\int_{\Omega}f(\zeta)\wedge P_{N}(z,\zeta)$, then $\overline{\partial}^{*}\mathcal{N}(g)$
is a solution of the equation $\overline{\partial}u=g$ satisfying
$\left\Vert \overline{\partial}^{*}\mathcal{N}(g)\right\Vert _{L_{s}^{2}(\Omega)}\leq C_{s,N,\gamma}\left\Vert f\right\Vert _{L^{1}(\Omega,\delta^{\gamma})}$.
\end{sllem}

Applying Sobolev lemma we immediately get:
\begin{sllem}
For $r\geq1$, $p\in\left[1,+\infty\right]$ and $-1<\gamma\leq N$,
if $f$ is a $\overline{\partial}$-closed $\left(0,r\right)$-form
with coefficients in $L^{1}(\Omega,\delta_{\Omega}^{\gamma})$ and
$g=\int_{\Omega}f(\zeta)\wedge P_{N}(z,\zeta)$, then $\overline{\partial}^{*}\mathcal{N}(g)$
is a solution of the equation $\overline{\partial}u=g$ satisfying
$\left\Vert \overline{\partial}^{*}\mathcal{N}(g)\right\Vert _{\mathcal{C}^{1}(\overline{\Omega})}\leq C\left\Vert f\right\Vert _{L^{1}\left(\Omega,\delta^{\gamma}\right)}$.
\end{sllem}

\medskip{}

Finally the proofs of our theorems are reduced to the proofs of good
estimates for the operator $T_{K}$ defined by
\begin{equation}
T_{K}:\,f\mapsto\int_{\Omega}f(\zeta)\wedge K_{N}^{1}(z,\zeta).\label{eq:operator_K1}
\end{equation}

To do it with some details we need to recall the anisotropic geometry
of $\Omega$ and the basic estimates given in \cite{CDMb}.

For $\zeta$ close to $\partial\Omega$ and $\varepsilon\leq\varepsilon_{0}$,
$\varepsilon_{0}$ small, define, for all unitary vector $v$,
\[
\tau\left(\zeta,v,\varepsilon\right)=\sup\left\{ c\mbox{ such that }\left|\rho\left(\zeta+\lambda v\right)-\rho(\zeta)\right|<\varepsilon,\,\forall\lambda\in\mathbb{C},\,\left|\lambda\right|<c\right\} .
\]

Let $\zeta$ and $\varepsilon$ be fixed. Then, an orthonormal basis
$\left(v_{1},v_{2},\ldots,v_{n}\right)$ is called \emph{$\left(\zeta,\varepsilon\right)$-extremal}
(or $\varepsilon$-\emph{extremal}, or simply \emph{extremal}) if
$v_{1}$ is the complex normal (to $\rho$) at $\zeta$, and, for
$i>1$, $v_{i}$ belongs to the orthogonal space of the vector space
generated by $\left(v_{1},\ldots,v_{i-1}\right)$ and minimizes $\tau\left(\zeta,v,\varepsilon\right)$
in that space. In association to an extremal basis, we denote
\[
\tau(\zeta,v_{i},\varepsilon)=\tau_{i}(\zeta,\varepsilon).
\]

Then we defined polydiscs $AP_{\varepsilon}(\zeta)$ by
\[
AP_{\varepsilon}(\zeta)=\left\{ z=\zeta+\sum_{k=1}^{n}\lambda_{k}v_{k}\mbox{ such that }\left|\lambda_{k}\right|\leq c_{0}A\tau_{k}(\zeta,\varepsilon)\right\} ,
\]
$c_{0}$ being sufficiently small, depending on $\Omega$, $P_{\varepsilon}(\zeta)$
being the corresponding polydisc with $A=1$ and we also define
\[
d(\zeta,z)=\inf\left\{ \varepsilon\mbox{ such that }z\in P_{\varepsilon}(\zeta)\right\} .
\]
The fundamental result here is that $d$ is a pseudo-distance which
means that, $\forall\alpha>0$, there exist constants $c(\alpha)$
and $C(\alpha)$ such that 
\begin{equation}
c(\alpha)P_{\varepsilon}(\zeta)\subset P_{\alpha\varepsilon}(\zeta)\subset C(\alpha)P_{\varepsilon}(\zeta)\mbox{ and }P_{c(\alpha)\varepsilon}(\zeta)\subset\alpha P_{\varepsilon}(\zeta)\subset P_{C(\alpha)\varepsilon}(\zeta).\label{eq:polydiscs-pseudodistance}
\end{equation}

For $\zeta$ close to $\partial\Omega$ and $\varepsilon>0$ small,
the basic properties of this geometry are (see \cite{Conrad_lineally_convex}
and \cite{CDMb}):
\begin{enumerate}
\item \label{geometry-1}Let $w=\left(w_{1},\ldots,w_{n}\right)$ be an
orthonormal system of coordinates centered at $\zeta$. Then
\[
\left|\frac{\partial^{\left|\alpha+\beta\right|}\rho(\zeta)}{\partial w^{\alpha}\partial\bar{w}^{\beta}}\right|\lesssim\frac{\varepsilon}{\prod_{i}\tau\left(\zeta,w_{i},\varepsilon\right)^{\alpha_{i}+\beta_{i}}},\,\left|\alpha+\beta\right|\geq1.
\]

\item \label{geometry-2}Let $\nu$ be a unit vector. Let $a_{\alpha\beta}^{\nu}(\zeta)=\frac{\partial^{\alpha+\beta}\rho}{\partial\lambda^{\alpha}\partial\bar{\lambda}^{\beta}}\left(\zeta+\lambda\nu\right)_{|\lambda=0}$.
Then
\[
\sum_{1\leq\left|\alpha+\beta\right|\leq2m}\left|a_{\alpha\beta}^{\nu}(\zeta)\right|\tau(\zeta,\nu,\varepsilon)^{\alpha+\beta}\simeq\varepsilon.
\]

\item \label{geometry-3}If $\left(v_{1},\ldots,v_{n}\right)$ is a $\left(\zeta,\varepsilon\right)$-extremal
basis and $\gamma=\sum_{1}^{n}a_{j}v_{j}\neq0$, then
\[
\frac{1}{\tau(\zeta,\gamma,\varepsilon)}\simeq\sum_{j=1}^{n}\frac{\left|a_{j}\right|}{\tau_{j}(\zeta,\varepsilon)}.
\]

\item \label{geometry-4}If $v$ is a unit vector then:

\begin{enumerate}
\item $z=\zeta+\lambda v\in P_{\varepsilon}(\zeta)$ implies $\left|\lambda\right|\lesssim\tau(\zeta,v,\varepsilon)$,
\item $z=\zeta+\lambda v$ with $\left|\lambda\right|\leq\tau(\zeta,v,\varepsilon)$
implies $z\in CP_{\varepsilon}(\zeta)$.
\end{enumerate}
\item \label{geometry-5}If $\nu$ is the unit complex normal, then $\tau(\zeta,v,\varepsilon)=\varepsilon$
and if $v$ is any unit vector and $\lambda\geq1$,
\begin{equation}
\lambda^{\nicefrac{1}{m}}\tau_{j}(\zeta,v,\varepsilon)\lesssim\tau_{j}(\zeta,v,\lambda\varepsilon)\lesssim\lambda\tau_{j}(\zeta,v,\varepsilon),\label{eq:comp-tau-epsilon-tau-lambda-epsilon}
\end{equation}
where $m$ is the type of $\Omega$.\end{enumerate}
\begin{sllem}
\label{lem:maj-deriv-rho-equiv-tho-i-z-zeta}For $z$ close to $\partial\Omega$,
$\varepsilon$ small and $\zeta\in P_{\varepsilon}(z)$, in the coordinate
system $\left(\zeta_{i}\right)$ associated to the $\left(z,\varepsilon\right)$-extremal
basis, we have:
\begin{enumerate}
\item $\left|\frac{\partial\rho}{\partial\zeta_{i}}(\zeta)\right|\lesssim\frac{\varepsilon}{\tau_{i}(z,\varepsilon)}$
(property (\ref{geometry-1}) of the geometry recalled above);
\item $\tau_{i}(\zeta,\varepsilon)\simeq\tau_{i}(z,\varepsilon)$ if $c_{0}$
is chosen sufficiently small.
\end{enumerate}
\end{sllem}

\medskip{}

We now recall the detailed expression of $K_{N}^{1}$ (\cite{CDMb}
sections 2.2 and 2.3):
\[
K_{N}^{1}(z,\zeta)=\sum_{k=n-r}^{n-1}C'_{k}\frac{\rho(\zeta)^{k+N}s\wedge\left(\partial_{\bar{\zeta}}Q\right)^{n-r}\wedge\left(\partial_{\bar{z}}Q\right)^{k+r-n}\wedge\left(\partial_{\bar{z}}s\right)^{n-k-1}}{\left|z-\zeta\right|^{2\left(n-k\right)}\left(\frac{1}{K_{0}}S(z,\zeta)+\rho(\zeta)\right)^{k+N}},
\]
where
\[
s(z,\zeta)=\sum_{i=1}^{n}\left(\overline{\zeta_{i}}-\overline{z_{i}}\right)d\left(\zeta_{i}-z_{i}\right)
\]
and
\[
Q(z,\zeta)=\frac{1}{K_{0}\rho(\zeta)}\sum_{i=1}^{n}Q_{i}(z,\zeta)d\left(\zeta_{i}-z_{i}\right)
\]
with
\[
S(z,\zeta)=\chi(z,\zeta)S_{0}(z,\zeta)-\left(1-\chi(z,\zeta)\right)\left|z-\zeta\right|^{2}=\sum_{i=1}^{n}Q_{i}(z,\zeta)\left(z_{i}-\zeta_{i}\right),
\]
$S_{0}$ being the holomorphic support function of Diederich-Fornaess
(see \cite{Diederich-Fornaess-Support-Func-lineally-cvx} or Theorem
2.2 of \cite{CDMb}) and $\chi$ a truncating function which is equal to $1$ when both $\left|z-\zeta\right|$ and
 $\delta_{\Omega}(\zeta)$ are small and $0$ if one of these expressions is large (see
the beginning of Section 2.2 of \cite{CDMb} for a precise definition). Recall that $K_{0}$
is chosen so that
\[
\Re\mathrm{e}\left(\rho(\zeta)+\frac{1}{K_{0}}S(z,\zeta)\right)<\frac{\rho(\zeta)}{2},
\]
that is
\begin{equation}
\left|\rho(\zeta)+\frac{1}{K_{0}}S(z,\zeta)\right|\gtrsim\left|\rho(\zeta)\right|.\label{eq:real-part-Q-zeta-z-plus-1}
\end{equation}

The following estimates of the expressions appearing in $K_{N}^{1}$
are basic (see \cite{CDMb}):
\begin{sllem}
\label{lem:lemma-maj-rho-S}For $\zeta\in P_{2\varepsilon}(z)\setminus P_{\varepsilon}(z)$,
we have:

\[
\left|\rho(\zeta)+\frac{1}{K_{0}}S(z,\zeta)\right|\gtrsim\varepsilon,\,\left(z,\zeta\right)\in\bar{\Omega}\times\bar{\Omega}.
\]

\end{sllem}

\begin{sllem}
\label{lem:Lemma-maj-deriv_rho_Q}For $z$ close to $\partial\Omega$,
$\varepsilon$ small and $\zeta\in P_{\varepsilon}(z)$, in the coordinate
system $\left(\zeta_{i}\right)$ associated to the $\left(z,\varepsilon\right)$-extremal
basis, we have:
\begin{enumerate}
\item $\left|Q_{i}(z,\zeta)\right|+\left|Q_{i}(\zeta,z)\right|\lesssim\frac{\varepsilon}{\tau_{i}(z,\varepsilon)}$
(see \cite{Diederich-Fischer_Holder-linally-convex});
\item $\left|\frac{\partial Q_{i}(z,\zeta)}{\partial\overline{\zeta_{j}}}\right|\lesssim\frac{\varepsilon}{\tau_{i}(z,\varepsilon)\tau_{j}(z,\varepsilon)}$
(see \cite{Diederich-Fischer_Holder-linally-convex});
\item $\left|\frac{\partial^{2}Q_{i}(z,\zeta)}{\partial\overline{\zeta}_{j}\partial z_{k}}\right|+\left|\frac{\partial^{2}Q_{i}(z,\zeta)}{\partial\overline{\zeta}_{j}\partial\overline{z}_{k}}\right|\lesssim\frac{\varepsilon}{\tau_{i}(z,\zeta)\tau_{j}(z,\zeta)\tau_{k}(z,\zeta)}$
(see \cite{Diederich-Fischer_Holder-linally-convex}).
\end{enumerate}
\end{sllem}

\bigskip{}

To simplify notations, we will now do the proofs of the theorems only
for $\left(0,1\right)$-forms, the general case of $\left(0,r\right)$-forms
being identical except for complications in the notations.

The preceding lemmas and the properties of the geometry easily give
the following estimates of the kernel $K_{N}^{1}$ (for $\left(0,1\right)$-forms):
\begin{sllem}
For $\varepsilon$ small enough and $z$ sufficiently close to the
boundary we have:

If $\zeta\in P_{\varepsilon}(z)$,
\[
\left|K_{N}^{1}(z,\zeta)\right|\lesssim\frac{\rho(\zeta)^{N-1}\left(\left|\rho(\zeta)\right|+\varepsilon\right)\varepsilon^{n-1}}{\prod_{i=1}^{n-1}\tau_{i}(z,\varepsilon)\left|\frac{1}{K_{0}}S(z,\zeta)+\rho(\zeta)\right|^{N+n-1}}\frac{1}{\left|z-\zeta\right|}.
\]

\end{sllem}

In particular:
\begin{sllem}
\label{lem:basic-estimates-kernel}For $\varepsilon$ small enough
and $z$ sufficiently close to the boundary:
\begin{enumerate}
\item If $\varepsilon\leq\delta_{\partial\Omega}(z)$, for $\zeta\in P_{\varepsilon}(z)$,
\[
\left|K_{N}^{1}(z,\zeta)\right|\lesssim\frac{1}{\prod_{i=1}^{n-1}\tau_{i}(z,\varepsilon)}\frac{1}{\left|z-\zeta\right|}.
\]

\item If $\zeta\in P_{2\varepsilon}(z)\setminus P_{\varepsilon}(z)$ or
$z\in P_{2\varepsilon}(\zeta)\setminus P_{\varepsilon}(\zeta)$ and
$k\leq N+n-1$,
\[
\left|K_{N}^{1}(z,\zeta)\right|\lesssim\frac{\left|\rho(\zeta)\right|^{k}}{\varepsilon^{k}}\frac{1}{\prod_{i=1}^{n-1}\tau_{i}}\frac{1}{\left|z-\zeta\right|},
\]
and
\[
\left|\nabla_{z}K_{N}^{1}(z,\zeta)\right|\lesssim\frac{\left|\rho(\zeta)\right|^{k}}{\varepsilon^{k+1}}\frac{1}{\prod_{i=1}^{n-1}\tau_{i}}\frac{1}{\left|z-\zeta\right|},
\]
where $\tau_{i}$ is either $\tau_{i}(z,\varepsilon)$ or $\tau_{i}(\zeta,\varepsilon)$.
\end{enumerate}
\end{sllem}

An elementary calculation shows that:
\begin{sllem}
For $z\in\Omega$, $\delta$ small and $0\leq\mu<1$,
\begin{equation}
\int_{P(z,\delta)}\frac{d\lambda(\zeta)}{\left|z-\zeta\right|^{1+\mu}}\lesssim\tau_{n}(z,\delta)^{1-\mu}\prod_{j=1}^{n-1}\tau_{j}^{2}(z,\delta),\label{eq:integral-mod(z-zeta)-inverse-power-1+mu}
\end{equation}
and, for $\alpha>0$,
\begin{equation}
\int_{P(\zeta,\delta)}\frac{\delta_{\Omega}^{\alpha-1}(z)}{\left|z-\zeta\right|}d\lambda(z)\lesssim\frac{\delta^{\alpha-1}}{\alpha}\tau_{n}(\zeta,\delta)\prod_{j=1}^{n-1}\tau_{j}^{2}(\zeta,\delta).\label{eq:integral-mod(z-zeta)-incerse-delta-alpha-1}
\end{equation}

\end{sllem}

\subsection{Proof of Theorem \ref{thm:d-bar-q-gamma'-p-gamma-lip}}
\begin{proof}[Proof of (1) of \thmref{d-bar-q-gamma'-p-gamma-lip}]
It is based on a version of a classical operator estimate which can
be found, for example, in Appendix B of the book of M. Range \cite{RM86}:
\begin{sllem}
Let $\Omega$ be a smoothly bounded domain in $\mathbb{C}^{n}$. Let
$\mu$ and $\nu$ be two positive measures on $\Omega$. Let $K$
be a measurable function on $\Omega\times\Omega$. Assume that there
exists a positive number $\varepsilon_{0}>0$, a positive constant
$C$ and a real number $s\geq1$ such that:
\begin{enumerate}
\item $\int_{\Omega}\left|K(z,\zeta)\right|^{s}\delta_{\Omega}^{-\varepsilon}(\zeta)d\mu(\zeta)\leq C\delta_{\Omega}^{-\varepsilon}(z)$,
\item $\int_{\Omega}\left|K(z,\zeta)\right|^{s}\delta_{\Omega}^{-\varepsilon}(z)d\nu(z)\leq C\delta_{\Omega}^{-\varepsilon}(\zeta)$,
\end{enumerate}
for all $\varepsilon\leq\varepsilon_{0}$, where $\delta_{\Omega}$
denotes the distance to the boundary of $\Omega$. Then the linear
operator $T$ defined by
\[
Tf(z)=\int_{\Omega}K(z,\zeta)f(\zeta)d\mu(\zeta)
\]
 is bounded from $L^{p}\left(\Omega,\mu\right)$ to $L^{q}\left(\Omega,\nu\right)$
for all $1\leq p,q<\infty$ such that $\frac{1}{q}=\frac{1}{p}+\frac{1}{s}-1$.\end{sllem}

\begin{proof}[Short proof]
This is exactly the proof given by M. Range in his book: let $\varepsilon$
be sufficiently small. Writing 
\[
Kf=\left(K^{s}f^{p}\delta_{\Omega}^{\varepsilon\frac{p-1}{p}q}(\zeta)\right)^{\nicefrac{1}{q}}\left(K^{1-\frac{s}{q}}\delta_{\Omega}^{-\varepsilon\frac{p-1}{p}}\right)f^{1-\frac{p}{q}},
\]
H\"{o}lder's inequality (with $\frac{1}{q}+\frac{p-1}{p}+\frac{s-1}{s}=1$)
gives
\begin{multline*}
\left|Tf(z)\right|\leq\left(\int_{\Omega}\left|K(z,\zeta)\right|^{s}\delta_{\Omega}^{\varepsilon\frac{p-1}{p}q}(\zeta)\left|f\right|^{p}(\zeta)d\mu(\zeta)\right)^{\nicefrac{1}{q}}\\
\left(\int_{\Omega}\left|K(z,\zeta)\right|^{s}\delta_{\Omega}^{-\varepsilon}(\zeta)\right)^{\frac{p-1}{p}}\left(\int_{\Omega}\left|f(\zeta)\right|^{p}d\mu(\zeta)\right)^{\frac{s-1}{s}}.
\end{multline*}
The first hypothesis of the lemma gives (for $\varepsilon\leq\varepsilon_{0}$)
\begin{multline*}
\left|Tf(z)\right|^{q}\leq C\left(\int_{\Omega}\left|K(z,\zeta)\right|^{s}\delta_{\Omega}^{\varepsilon\frac{p-1}{p}q}(\zeta)\delta_{\Omega}^{-\varepsilon\frac{p-1}{p}q}(z)\left|f\right|^{p}(\zeta)d\mu(\zeta)\right)\\
\left(\int_{\Omega}\left|f(\zeta)\right|^{p}d\mu(\zeta)\right)^{q\frac{s-1}{s}}.
\end{multline*}
Integration with respect to the measure $d\nu(z)$ gives (using the
second hypothesis of the lemma with $\varepsilon\frac{p-1}{p}q\leq\varepsilon_{0}$)
\[
\int_{\Omega}\left|Tf(z)\right|^{q}d\nu(z)\leq C^{2}\left(\int_{\Omega}\left|f\right|^{p}d\mu\right)^{\nicefrac{q}{p}}.
\]

\end{proof}

Applying this lemma to the operator $T_{K}$ (formula (\ref{eq:operator_K1}))
with $\mu=\delta_{\Omega}^{\gamma}d\lambda$ and $\nu=\delta_{\Omega}^{\gamma'}d\lambda$,
the required estimates on $K_{N}^{1}$ are summarized in the following
Lemma:
\begin{sllem}
\label{lem:estimates-kernel}\quad\mynobreakpar
\begin{enumerate}
\item Let $\mu_{0}=\frac{1}{m(\gamma+n)+1}$. Then for $-1<\gamma<N-1$
and $\varepsilon>0$ sufficiently small,
\[
\int_{\Omega}\left|K_{N}^{1}\left(z,\zeta\right)\right|^{1+\mu_{0}}\delta_{\Omega}(\zeta)^{-\mu_{0}\gamma-\varepsilon}d\lambda(\zeta)\lesssim\delta_{\Omega}(z)^{-\varepsilon}.
\]

\item Let $\mu_{0}=\frac{1-m(\gamma-\gamma')}{m(\gamma+n)+1}$. Then for
$-1<\gamma<N-1$ and $\varepsilon>0$ sufficiently small,
\[
\int_{\Omega}\left|K_{N}^{1}\left(z,\zeta\right)\right|^{1+\mu_{0}}\frac{\delta_{\Omega}(z)^{\gamma'-\varepsilon}}{\delta_{\Omega}(\zeta)^{(1+\mu_{0})\gamma}}d\lambda(z)\lesssim\delta_{\Omega}(\zeta)^{-\varepsilon}.
\]

\end{enumerate}
\end{sllem}

We now prove this last lemma.
\begin{proof}[Proof of (1) of \lemref{estimates-kernel}]
$K_{N}^{1}$ being bounded, uniformly in $\left(z,\zeta\right)$,
outside $P_{\varepsilon_{0}}(z)$, it is enough to prove that
\[
\int_{P_{\varepsilon_{0}}(z)}\left|K_{N}^{1}(z,\zeta)\right|^{1+\mu_{0}}\delta_{\Omega}^{-\gamma\mu_{0}-\varepsilon}(\zeta)d\lambda(\zeta)\lesssim\delta_{\Omega}^{-\varepsilon}(z)
\]
for $\varepsilon_{0}$ and $\varepsilon$ sufficiently small. As this
is trivial if $z$ is far from the boundary, we assume that $z$ is
sufficiently close to $\partial\Omega$.

Let $A(z,\zeta)=K_{N}^{1}(z,\zeta)\left|z-\zeta\right|$. If $\zeta\in P\left(z,\delta_{\Omega}(z)\right)$
then $\delta_{\Omega}(z)\simeq\delta_{\Omega}(\zeta)$ and, by (2)
of \lemref{basic-estimates-kernel},
\begin{equation}
\left|A(z,\zeta)\right|^{1+\mu_{0}}\delta_{\Omega}^{-\gamma\mu_{0}-\varepsilon}(\zeta)\lesssim\delta_{\Omega}(z)^{-\mu_{0}(\gamma+n)-\varepsilon}\prod_{j=1}^{n-1}\tau_{j}^{2}\left(z,\delta_{\Omega}(z)\right).\label{eq:estimate-Z(z,zeta)-around-z}
\end{equation}
Thus, by (\ref{eq:integral-mod(z-zeta)-inverse-power-1+mu}), we get
\begin{eqnarray*}
\int_{P(z,\delta_{\Omega}(z))}\left|K_{N}^{1}(z,\zeta)\right|^{1+\mu_{0}}\delta_{\Omega}^{-\gamma\mu_{0}-\varepsilon}(\zeta)d\lambda(\zeta) & \lesssim & \delta_{\Omega}(z)^{-\mu_{0}(\gamma+n)-\varepsilon+\frac{1-\mu_{0}}{m}}\\
 & = & \delta_{\Omega}(z)^{-\varepsilon}.
\end{eqnarray*}

Now, let $\zeta\in P_{2^{i}\delta_{\Omega}(z)}(z)\setminus P_{2^{(i+1)}\delta_{\Omega}(z)}(z)$,
if $N$ is sufficiently large ($N\geq\gamma+n+1$), by (3) of \lemref{basic-estimates-kernel},
we have
\[
\left|A(z,\zeta)\right|^{1+\mu_{0}}\delta_{\Omega}^{-\gamma\mu_{0}-\varepsilon}(\zeta)\lesssim\left(2^{i}\delta_{\Omega}(z)\right)^{-\mu_{0}(\gamma+n)-\varepsilon}\prod_{j=1}^{n-1}\tau_{j}^{2}\left(z,2^{i}\delta_{\Omega}(z)\right)
\]

which gives ((\ref{eq:integral-mod(z-zeta)-inverse-power-1+mu}))
\begin{eqnarray*}
\int_{P^{i}(z)}\left|K_{N}^{1}(z,\zeta)\right|^{1+\mu_{0}}\delta_{\Omega}^{-\gamma\mu_{0}-\varepsilon}(\zeta)d\lambda(\zeta) & \lesssim & \left(2^{i}\delta_{\Omega}(z)\right)^{-\mu_{0}(\gamma+n)-\varepsilon+\frac{1-\mu_{0}}{m}}\\
 & = & \delta_{\Omega}(z)^{-\varepsilon}\left(2^{i}\right)^{-\varepsilon},
\end{eqnarray*}

finishing the proof.
\end{proof}

\begin{proof}[Proof of (2) of \lemref{estimates-kernel}]
As in the preceding proof we have to show that
\[
\int_{P_{\varepsilon_{0}}(\zeta)}\left|\frac{K_{N}^{1}(z,\zeta)}{\delta_{\Omega}(\zeta)^{\gamma}}\right|^{1+\mu_{0}}\delta_{\Omega}(z)^{\gamma'-\varepsilon}d\lambda(z)\lesssim\delta_{\Omega}(\zeta)^{-\varepsilon}.
\]

If $z\in P\left(\zeta,\delta_{\Omega}(z)\right)$ then $\delta_{\Omega}(\zeta)\simeq\delta_{\Omega}(z)$,
the estimate (\ref{eq:estimate-Z(z,zeta)-around-z}), which is still
valid replacing $\tau_{j}\left(z,\delta_{\Omega}(z)\right)$ by $\tau_{j}\left(\zeta,\delta_{\Omega}(\zeta)\right)$
(\lemref{maj-deriv-rho-equiv-tho-i-z-zeta}), and (\ref{eq:integral-mod(z-zeta)-inverse-power-1+mu})
(interchanging the roles of $z$ and $\zeta$), we immediately get
\begin{eqnarray*}
\int_{P\left(\zeta,\delta_{\Omega}(\zeta)\right)}\left|\frac{K_{N}^{1}(z,\zeta)}{\delta_{\Omega}(\zeta)^{\gamma}}\right|^{1+\mu_{0}}\delta_{\Omega}(z)^{\gamma'-\varepsilon}d\lambda(z) & \lesssim & \delta_{\Omega}(\zeta)^{-\mu_{0}(\gamma+n)-(\gamma-\gamma')+\frac{1-\mu_{0}}{m}-\varepsilon}\\
 & = & \delta_{\Omega}(\zeta)^{-\varepsilon}.
\end{eqnarray*}

Assume now $z\in P^{i}(\zeta)=P_{2^{i}\delta_{\Omega}(\zeta)}(\zeta)\setminus P_{2^{(i+1)}\delta_{\Omega}(\zeta)}(\zeta)$.

If $\gamma'-\varepsilon\geq0$, using $\delta_{\Omega}(z)\lesssim2^{i}\delta_{\Omega}(\zeta)$,
(3) of \lemref{basic-estimates-kernel} and (\ref{eq:integral-mod(z-zeta)-inverse-power-1+mu})
give
\begin{eqnarray*}
\int_{P^{i}(\zeta)}\left|\frac{K_{N}^{1}(z,\zeta)}{\delta_{\Omega}(\zeta)^{\gamma}}\right|^{1+\mu_{0}}\delta_{\Omega}(z)^{\gamma'-\varepsilon}d\lambda(z) & \lesssim & \left(2^{i}\delta_{\Omega}(\zeta)\right)^{-\mu_{0}(\gamma+n)-(\gamma-\gamma')+\frac{1-\mu_{0}}{m}-\varepsilon}\\
 & = & \delta_{\Omega}(\zeta)^{-\varepsilon}\left(2^{i}\right)^{-\varepsilon},
\end{eqnarray*}

finishing the proof in that case.

If $-1<\gamma'-\varepsilon$$\leq0$, as
\[
\int_{P^{i}(\zeta)}\left|\frac{\delta_{\Omega}(z)}{\delta_{\Omega}(\zeta)}\right|^{\gamma'-\varepsilon}\frac{d\lambda(z)}{\left|z-\zeta\right|}\lesssim_{\gamma'-\epsilon}\tau_{n}\left(\zeta,2^{i}\delta_{\Omega}(\zeta)\right)\prod_{j=1}^{n-1}\tau_{j}^{i}\left(\zeta,2^{i}\delta_{\Omega}(\zeta)\right),
\]

the proof is done as before using (3) of \lemref{basic-estimates-kernel}.
\end{proof}
The proof of (1) of \thmref{d-bar-q-gamma'-p-gamma-lip} is now complete.
\end{proof}
\medskip{}

\begin{proof}[Proof of (2) and (3) of \thmref{d-bar-q-gamma'-p-gamma-lip}]
By the Hardy-Littlewood lemma we have to prove the two following
inequalities:
\begin{itemize}
\item if $p=m(\gamma+n)+2$, $\nabla_{z}\left(\int_{\Omega}f(\zeta)\wedge K_{N}^{1}(z,\zeta)\right)\lesssim\delta_{\Omega}(z)^{-1}$,
\item if $p>m(\gamma+n)+2$, $\nabla_{z}\left(\int_{\Omega}f(\zeta)\wedge K_{N}^{1}(z,\zeta)\right)\lesssim\delta_{\Omega}(z)^{\alpha-1}$.
\end{itemize}
Then, using H\"{o}lder's inequality these two estimates are consequences
of the following lemma:
\begin{sllem}
Let $p\geq m(\gamma+n)+2$, $p'$ the conjugate of $p$ (i.e. $\frac{1}{p}+\frac{1}{p'}=1$)
and let $\alpha=\frac{1}{m}\left[1-\frac{m(\gamma+n)+2}{p}\right]$.
Then
\[
\int_{\Omega}\left|\nabla_{z}K_{N}^{1}(z,\zeta)\right|^{p'}\delta_{\Omega}(\zeta)^{-\gamma\nicefrac{p'}{p}}\lesssim\delta_{\Omega}(z)^{p'(\alpha-1)}.
\]
\end{sllem}

\begin{proof}[Proof of the lemma]
Denote $p'=1+\eta$ so that $\nicefrac{p'}{p}=\eta$ and $\nicefrac{1}{p}=\frac{\eta}{1+\eta}$.
By the basic estimates of $K_{N}^{1}$ (and the fact that $-\frac{\gamma p'}{p}>-1$)
it suffices to estimate the above integral when the domain of integration
is reduced to $P(z,\varepsilon_{0})$.

Assume first that $\zeta\in P(z,2^{i}\delta_{\Omega}(z))\setminus P(z,2^{i-1}\delta_{\Omega}(z))$.
Then, by (3) of \lemref{basic-estimates-kernel}, we have
\[
\left|\nabla_{z}K_{N}^{1}(z,\zeta)\right|\lesssim\frac{\left|\delta_{\Omega}(\zeta)\right|^{\nicefrac{\gamma}{p}}}{\left(2^{i}\delta_{\Omega}(z)\right)^{1+\nicefrac{\gamma}{p}}}\frac{1}{\prod_{j=1}^{n-1}\tau_{j}^{2}\left(z,2^{i}\delta_{\Omega}(z)\right)}\frac{1}{\left|z-\zeta\right|},
\]
and by (\ref{eq:integral-mod(z-zeta)-inverse-power-1+mu}), we get
\begin{eqnarray*}
\int_{P(z,2^{i}\delta_{\Omega}(z))\setminus P(z,2^{i-1}\delta_{\Omega}(z))}\left|\nabla_{z}K_{N}^{1}(z,\zeta)\right|^{p'}\delta_{\Omega}(\zeta)^{-\gamma\nicefrac{p'}{p}} & d\lambda(\zeta)\lesssim & \left(2^{i}\delta_{\Omega}(z)\right)^{-n\eta+\frac{1-\eta}{m}-p'-\gamma\eta}\\
 & = & \left(2^{i}\right)^{p'(\alpha-1)}\delta_{\Omega}(z)^{p'(\alpha-1)}.
\end{eqnarray*}

Assume now that $\zeta\in P(z,2^{-(i-1)}\delta_{\Omega}(z))\setminus P(z,2^{-i}\delta_{\Omega}(z))$.
Then, by (3) of \lemref{basic-estimates-kernel}, we have
\[
\left|\nabla_{z}K_{N}^{1}(z,\zeta)\right|\lesssim\frac{1}{\prod_{j=1}^{n-1}\tau_{j}^{2}\left(z,\delta_{\Omega}(z)\right)}\frac{1}{\left|z-\zeta\right|}\frac{1}{2^{-i}\delta_{\Omega}(z)},
\]
and, by (\ref{eq:integral-mod(z-zeta)-inverse-power-1+mu}), we have
\[
\int_{P(z,2^{-(i-1)}\delta_{\Omega}(z))\setminus P(z,2^{-i}\delta_{\Omega}(z))}\frac{d\lambda(\zeta)}{\left|z-\zeta\right|^{1+\eta}}\lesssim\left(2^{-i}\right)^{2}\prod_{j=1}^{n-1}\tau_{j}^{2}\left(z,\delta_{\Omega}(z)\right)\tau_{n}^{\frac{1-\eta}{m}}\left(z,\delta_{\Omega}(z)\right).
\]
Thus, as $\delta_{\Omega}(z)\simeq\delta_{\Omega}(\zeta)$, we get
\[
\int_{P(z,2^{-(i-1)}\delta_{\Omega}(z))\setminus P(z,2^{-i}\delta_{\Omega}(z))}\left|\nabla_{z}K_{N}^{1}(z,\zeta)\right|^{p'}\delta_{\Omega}(\zeta)^{-\gamma\nicefrac{p'}{p}}d\lambda(\zeta)\lesssim2^{-i}\delta_{\Omega}(z)^{p'(\alpha-1)},
\]
finishing the proof of the lemma.
\end{proof}
The proofs of (2) and (3) of \thmref{d-bar-q-gamma'-p-gamma-lip}
are complete.
\end{proof}
The proof of \thmref{d-bar-q-gamma'-p-gamma-lip} is now complete.

\bigskip{}

\subsection{\label{sec:Proof-of-thm-est-d-bar-Nev}Proof of Theorem \ref{thm:d-bar-for-Nev}}

\quad

First we briefly recall the definition of the anisotropic norm $\left\Vert .\right\Vert _{k}$
given in \cite{CDMb}: for $z$ close to the boundary,
\[
\left\Vert f(z)\right\Vert _{k}=\sup_{\left\Vert v_{i}\right\Vert =1}\frac{\left|\left\langle f;v_{1},\ldots,v_{q}\right\rangle (z)\right|}{\sum_{i=1}^{q}k\left(z,v_{i}\right)},
\]
where $k\left(z,v\right)=\frac{\delta_{\Omega}(z)}{\tau\left(z,v,\delta_{\Omega}(z)\right)}$.
The estimate needed for the operator (\ref{eq:operator_K1}) to prove
the theorem is
\begin{sllem}
For $\alpha>0$, we have
\[
\int_{\Omega}\delta_{\Omega}^{\alpha-1}(z)\left|K_{N}^{1}(z,\zeta)\wedge f(\zeta)\right|d\lambda(z)\lesssim\frac{1}{\alpha}\delta_{\Omega}^{\alpha}(\zeta)\left\Vert f(\zeta)\right\Vert _{k}.
\]
\end{sllem}

\begin{proof}
As before, we consider only the case of $\left(0,1\right)$-forms
$f$ and we assume $\zeta$ sufficiently close to the boundary.

Denote $Q_{0}(\zeta)=P\left(\zeta,\delta_{\Omega}(\zeta)\right)$
and $Q_{i}(\zeta)=P\left(\zeta,2^{i}\delta_{\Omega}(\zeta)\right)\setminus P\left(\zeta,2^{i-1}\delta_{\Omega}(\zeta)\right)$,
$i=1,2,\ldots$ and let us prove
\[
\int_{Q_{i}}\delta_{\Omega}^{\alpha-1}(z)\left|K_{N}^{1}(z,\zeta)\wedge f(\zeta)\right|d\lambda(z)\lesssim\frac{1}{2^{i}}\frac{1}{\alpha}\delta_{\Omega}^{\alpha}(\zeta)\left\Vert f(\zeta)\right\Vert _{k}.
\]

Expressing the forms $K_{N}^{1}(z,\zeta)$ and $f(\zeta)$in the coordinate
system $\left(\zeta_{i}\right)_{i}$ associated to a $\left(\zeta,2^{i}\delta_{\Omega}(\zeta)\right)$-extremal
basis, we have to show that, for $i=0,1,\ldots$ and $1\leq l\leq n$,
\[
\int_{Q_{i}(\zeta)}\delta_{\Omega}^{\alpha-1}(z)\left|K_{N}^{1}(z,\zeta)\wedge d\overline{\zeta}_{l}\right|d\lambda(z)\lesssim\frac{1}{2^{i}}\frac{1}{\alpha}\delta_{\Omega}^{\alpha}(\zeta)\left\Vert d\overline{\zeta}_{l}\right\Vert _{k}.
\]

First, we remark that $K_{N}^{1}(z,\zeta)\wedge d\overline{\zeta}_{l}$
is a sum of expressions of the form $\frac{W}{D}$ where
\[
D(\zeta,z)=\left|z-\zeta\right|^{2}\left(\frac{1}{K_{0}}S(z,\zeta)+\rho(\zeta)\right)^{n+N-1},
\]
and,
\[
W=\left(\overline{\zeta}_{m}-\overline{z}_{m}\right)\rho^{N}(\zeta)\prod_{k=1}^{n-1}\frac{\partial Q_{i_{k}}(z,\zeta)}{\partial\overline{\zeta}_{j_{k}}}\bigwedge_{i=1}^{n}\left(d\zeta_{i}\wedge d\overline{\zeta_{i}}\right)
\]
or
\[
W=\left(\overline{\zeta}_{m}-\overline{z}_{m}\right)\rho^{N-1}(\zeta)\frac{\partial\rho(\zeta)}{\partial\overline{\zeta}_{j_{k_{0}}}}Q_{i_{k_{0}}}(\zeta,z)\prod_{\SU{1\leq k\leq n-1\AS k\neq k_{0}}}\frac{\partial Q_{i_{k}}(z,\zeta)}{\partial\overline{\zeta}_{j_{k}}}\bigwedge_{i=1}^{n}\left(d\zeta_{i}\wedge d\overline{\zeta_{i}}\right),
\]
with $\left\{ i_{1,\ldots,i_{n-1},m}\right\} =\left\{ j_{1},\ldots,j_{n-1},l\right\} =\left\{ 1,\ldots,n\right\} $.

Then, using \lemref{Lemma-maj-deriv_rho_Q} (and the properties of
the geometry) we obtain the following estimates:

For $z\in Q_{0}$, $\left|K_{N}^{1}(z,\zeta)\wedge d\overline{\zeta}_{l}\right|$
is bounded by a sum of expressions of the form 
\[
\frac{1}{\prod_{j=1}^{n}\tau_{j}^{2}\left(\zeta,\delta_{\Omega}(\zeta)\right)}\tau_{m}\left(\zeta,\delta_{\Omega}(\zeta)\right)\tau_{l}\left(\zeta,\delta_{\Omega}(\zeta)\right)\frac{1}{\left|z-\zeta\right|}.
\]
This gives (using (\ref{eq:integral-mod(z-zeta)-incerse-delta-alpha-1}))
\begin{eqnarray*}
\int_{Q_{0}}\delta_{\Omega}^{\alpha-1}(z)\left|K_{N}^{1}(z,\zeta)\wedge d\overline{\zeta}_{l}\right|d\lambda(z) & \lesssim & \frac{\delta_{\Omega}^{\alpha-1}(\zeta)}{\alpha}\tau_{l}\left(\zeta,\delta_{\Omega}(\zeta)\right)\\
 & = & \frac{\delta_{\Omega}^{\alpha}(\zeta)}{\alpha}\frac{\tau_{l}\left(\alpha,\delta_{\Omega}(\zeta)\right)}{\delta_{\Omega}(\zeta)}\\
 & \leq & \frac{\delta_{\Omega}^{\alpha}(\zeta)}{\alpha}\left\Vert d\overline{\zeta}_{l}\right\Vert _{k}.
\end{eqnarray*}

For $z\in Q_{i}$, $\left|K_{N}^{1}(z,\zeta)\wedge d\overline{\zeta}_{l}\right|$
is bounded by a sum of expressions of the form
\[
\frac{\delta_{\Omega}(\zeta)}{2^{i}\delta_{\Omega}(\zeta)}\frac{1}{\prod_{j=1}^{n}\tau_{j}^{2}\left(\zeta,\delta_{\Omega}(\zeta)\right)}\tau_{m}\left(\zeta,\delta_{\Omega}(\zeta)\right)\tau_{l}\left(\zeta,\delta_{\Omega}(\zeta)\right)\frac{1}{\left|z-\zeta\right|},
\]
giving, for $N\geq\alpha+2$,
\begin{eqnarray*}
\int_{Q_{i}}\delta_{\Omega}^{\alpha-1}(z)\left|K_{N}^{1}(z,\zeta)\wedge d\overline{\zeta}_{l}\right|d\lambda(z) & \lesssim & \left[\frac{\delta_{\Omega}(\zeta)}{2^{i}\delta_{\Omega}(\zeta)}\right]^{\alpha+1}\frac{\left(2^{i}\delta_{\Omega}(\zeta)\right)^{\alpha}}{\alpha}\frac{\tau_{l}\left(\zeta,2^{i}\delta_{\Omega}(\zeta)\right)}{2^{i}\delta_{\Omega}(\zeta)}\\
 & \lesssim & \frac{1}{2^{i}}\frac{\delta_{\Omega}^{\alpha}(\zeta)}{\alpha}\frac{\tau_{l}\left(\zeta,2^{i}\delta_{\Omega}(\zeta)\right)}{2^{i}\delta_{\Omega}(\zeta)}\\
 & \lesssim & \frac{1}{2^{i}}\frac{\delta_{\Omega}^{\alpha}(\zeta)}{\alpha}\frac{\tau_{l}\left(\zeta,\delta_{\Omega}(\zeta)\right)}{\delta_{\Omega}(\zeta)}\\
 & \leq & \frac{1}{2^{i}}\frac{\delta_{\Omega}^{\alpha}(\zeta)}{\alpha}\left\Vert d\overline{\zeta}_{l}\right\Vert _{k},
\end{eqnarray*}
the penultimate inequality coming from property (\ref{eq:comp-tau-epsilon-tau-lambda-epsilon})
of the geometry.

The lemma is proved and so is \thmref{d-bar-for-Nev}.
\end{proof}

\section{Proof of Theorem \ref{thm:estimates-bergman}}

We use the method developed in \cite{CPDY} for the proofs of theorems
2.1 and 2.3 of that paper.

In \cite{CDM} we prove, in particular, the following result: let
$g$ be a gauge of $D$ and $\rho_{0}=g^{4}e^{1-\nicefrac{1}{g}}-1$
then:
\begin{stthm}[Theorem 2.1 of \cite{CDM}]
\label{thm:regularity-P-omega-0}Let $\omega_{0}=\left(-\rho_{0}\right)^{r}$,
$r$ being a non negative rational number, and let $P_{\omega_{0}}$
be the Bergman projection of the Hilbert space $L^{2}\left(\Omega,\omega_{0}\right)$.
Then, for $p\in\left]1,+\infty\right[$ and $1\leq\beta\leq p\left(r+1\right)-1$,
$P_{\omega_{0}}$ maps continuously the space $L^{p}\left(D,\delta_{D}^{\beta}\right)$
into itself and, for $\alpha>0$, $P_{\omega_{0}}$ maps continuously
the lipschitz space $\Lambda_{\alpha}(D)$ into itself.
\end{stthm}

If $\omega$ is as in \thmref{estimates-bergman} then there exists
a strictly positive $\mathcal{C}^{1}$ function in $\overline{D}$,
$\varphi$, such that $\omega=\varphi\omega_{0}$. Then we compare
the regularity of $P_{\omega_{0}}$ and $P_{\omega}$ using the following
formula (Proposition 3.1 of \cite{CPDY}): for $u\in L^{2}\left(D,\omega\right)$,
\[
\varphi P_{\omega}(u)=P_{\omega_{0}}(\varphi u)+\left(\mathrm{Id}-P_{\omega_{0}}\right)\circ A\left(P_{\omega}(u)\wedge\overline{\partial}\varphi\right),
\]
where $A$ is any operator solving the $\overline{\partial}$-equation
for $\overline{\partial}$-closed forms in $L^{2}\left(D,\omega\right)$.

We first show that $P_{\omega}$ maps continuously $L^{p}\left(\Omega,\delta_{\Omega}^{r}\right)$
into itself. Let $f\in L^{p}\left(D,\delta_{\Omega}^{r}\right)$,
$p\in\left[2,+\infty\right[$. For $A$ we choose the operator $T$
of \propref{d-bar-gain-exponent} with $\gamma=r$, and we choose
$0<\varepsilon\leq\varepsilon_{0}$, $\varepsilon_{0}$ as in \propref{d-bar-gain-exponent},
such that there exists an integer $N$ such that $p=2+N\varepsilon$.
Let us prove, by induction, that $P_{\omega}(f)\in L^{2+k\varepsilon}\left(D,\delta_{D}^{r}\right)$
for $k=0,\ldots,N$.

Assume this is true for $0\leq k<N$. Then by \propref{d-bar-gain-exponent},
\[
A\left(P_{\omega}(f)\wedge\overline{\partial}\varphi\right)\in L^{2+(k+1)\varepsilon}\left(D,\delta_{D}^{r}\right)
\]
and, by \thmref{regularity-P-omega-0}, 
\[
\left(\mathrm{Id}-P_{\omega_{0}}\right)\circ A\left(P_{\omega}(u)\wedge\overline{\partial}\varphi\right)\in L^{2+(k+1)\varepsilon}\left(D,\delta_{D}^{r}\right).
\]
As $\varphi$ is continuous and strictly positive we get $P_{\omega}(f)\in L^{2+(k+1)\varepsilon}\left(D,\delta_{D}^{r}\right)$.

Thus, $P_{\omega}$ maps $L^{p}\left(D,\delta_{D}^{r}\right)$ into
it self for $p\in\left[2,+\infty\right[$. The same result for $p\in\left]1,2\right]$
follows because $P_{\omega}$ is self-adjoint.

\medskip{}

To prove that $P_{\omega}$ maps $L^{p}\left(D,\delta_{D}^{\beta}\right)$
for $-1<\beta\leq r$, we use a similar induction argument using \propref{d-bar-gain-weight}
instead of \propref{d-bar-gain-exponent}:

For $A$ we choose now the operator $T$ of \propref{d-bar-gain-weight}
with $\gamma=r$, and $0<\varepsilon\leq\varepsilon_{0}$, $\varepsilon_{0}$
as in \propref{d-bar-gain-weight} such that there exists an integer
$L$ such that $\beta=r-L\varepsilon$. For $f\in L^{p}\left(D,\delta_{D}^{\beta}\right)$,
assume $P_{\omega}(f)\in L^{2}\left(D,\delta_{D}^{r-l\varepsilon}\right)$,
$0\leq l<L$. Then, \propref{d-bar-gain-weight} and \thmref{regularity-P-omega-0}
imply $\left(\mathrm{Id}-P_{\omega_{0}}\right)\circ A\left(P_{\omega}(u)\wedge\overline{\partial}\varphi\right)\in L^{p}\left(D,\delta_{D}^{r-(l+1)\varepsilon}\right)$
which gives $P_{\omega}(f)\in L^{p}\left(D,\delta_{D}^{r-(l+1)\varepsilon}\right)$.
By induction this gives $P_{\omega}(f)\in L^{p}\left(D,\delta_{D}^{\beta}\right)$,
concluding the proof of (1) of the theorem.

The proof of (2) of the theorem is now easily done: assume $u\in\Lambda_{\alpha}(D)$,
$0<\alpha\leq\nicefrac{1}{m}$. Let $p\leq+\infty$ such that $\alpha=\frac{1}{m}\left[1-\frac{m(r+n)+2}{p}\right]$.
By part (1), $P_{\omega}(u)\in L^{p}(D,\delta_{D}^{r})$, by (3) of
\thmref{d-bar-q-gamma'-p-gamma-lip}, $A\left(P_{\omega}(u)\wedge\overline{\partial}\varphi\right)\in\Lambda_{\alpha}(D)$
($A$ being the operator $T$), and, by \thmref{regularity-P-omega-0},
$\left(\mathrm{Id}-P_{\omega_{0}}\right)\circ A\left(P_{\omega}(u)\wedge\overline{\partial}\varphi\right)\in\Lambda_{\alpha}(D)$
concluding the proof.
\begin{rem*}
\quad\mynobreakpar
\begin{enumerate}
\item The restriction $-1<\beta\leq r$ in \ref{thm:estimates-bergman}
(instead of $0<\beta+1\leq p(r+1)$ in \cite{CDM}) is due to the
method because if $f\in L^{p}\left(D,\delta_{D}^{\beta}\right)$ with
$\beta>r$, a priori $P_{\omega}(f)$ does not exists.
\item The restriction $r\in\mathbb{Q}_{+}$ is not natural and it is very
probable that \thmref{estimates-bergman} is true with $r\in\mathbb{R}_{+}$.
To get that with our method we should first prove the result of \thmref{regularity-P-omega-0}
for $r$ a non negative real number. Looking at the proof in \cite{CDM},
this should be done proving point-wise estimates of the Bergman kernel
of a domain $\widetilde{D}$ of the form
\[
\widetilde{D}=\left\{ (z,w)\in\mathbb{C}^{n+m}\mbox{ such that }\rho_{0}(z)+\sum\left|w_{i}\right|^{2q_{i}}<0\right\} ,
\]
with $q_{i}$ large \emph{real} numbers such that $\sum\nicefrac{1}{q_{i}}=r$.
The difficulty here being that $\widetilde{D}$ is no more $\mathcal{C}^{\infty}$-smooth
and thus the machinery induced by the finite type cannot be used.
\end{enumerate}
\end{rem*}

\bibliographystyle{amsalpha}

\providecommand{\bysame}{\leavevmode\hbox to3em{\hrulefill}\thinspace}
\providecommand{\MR}{\relax\ifhmode\unskip\space\fi MR }
\providecommand{\MRhref}[2]{%
  \href{http://www.ams.org/mathscinet-getitem?mr=#1}{#2}
}
\providecommand{\href}[2]{#2}

\end{document}